\documentclass[11pt,a4paper,twoside]{article}
\usepackage{latexsym,amsfonts,amsmath, amsthm, amssymb}
\usepackage{fullpage}
\usepackage{xcolor,bm, bbm}
\usepackage[utf8x]{inputenc}
\usepackage{array}
\usepackage{mathrsfs}

\usepackage{fourier}
\newcommand{\R}{\mathbb R}
\newcommand{\N}{\mathbb N}

\newcommand{\E}{\mathbb E}
\newcommand{\Pro}{\mathbb P}

\newcommand{\Var}{\mathrm{Var}}
\newcommand{\Uni}{\mathrm{Unif}}

\newcommand{\vol}{\mathrm{vol}}

\def\dint{\textup{d}}
\newcommand{\SSS}{\ensuremath{{\mathbb S}}}

\newcommand{\B}{\ensuremath{{\mathbb B}}}

\newcommand{\eps}{\varepsilon}

\newcommand{\ii}{{\rm{i}}}
\newcommand{\eee}{{\rm e}}

\DeclareMathOperator{\sinc}{sinc}

\newtheorem{thm}{Theorem}[section]

\newtheorem{lemma}[thm]{Lemma}
\newtheorem{df}[thm]{Definition}
\newtheorem{proposition}[thm]{Proposition}

\newtheorem{thmalpha}{Theorem}

{
\theoremstyle{definition}

}
\allowdisplaybreaks
\begin{document}

\title{\bf Projections of the uniform distribution on the cube \\
-- a large deviation perspective}

\medskip

\author{Samuel G.~G.~Johnston, Zakhar Kabluchko and Joscha Prochno}

%\thanks{~}

%\keywords{}
%\subjclass{}
%% NB There should be only one primary classification, and zero or
%more secondary classifications.

\date{}

\maketitle

\begin{abstract}
\small
Let $\Theta^{(n)}$ be a random vector uniformly distributed on the unit sphere
$\SSS^{n-1}$ in $\mathbb R^n$.  Consider the projection of the uniform distribution on the cube $[-1,1]^n$ to the line spanned by $\Theta^{(n)}$. The projected distribution is the random probability measure $\mu_{\Theta^{(n)}}$ on $\mathbb R$ given by
\[
\mu_{\Theta^{(n)}}(A) := \frac 1 {2^n} \int_{[-1,1]^n}  \mathbb 1\{\langle u, \Theta^{(n)} \rangle \in A\}
\,\dint u,
\]
for Borel subets $A$ of $\mathbb{R}$.  
It is well known that, with probability $1$,  the sequence of random probability measures $\mu_{\Theta^{(n)}}$ converges weakly to the centered Gaussian distribution with variance $1/3$. We prove a large deviation principle for the sequence $\mu_{\Theta^{(n)}}$ on the space of probability measures on $\mathbb R$ with speed $n$. The (good) rate function is explicitly given by $I(\nu(\alpha)) := - \frac{1}{2} \log ( 1 - ||\alpha||_2^2)$ whenever $\nu(\alpha)$ is the law of a random variable of the form
\begin{align*}
\sqrt{1 - ||\alpha||_2^2 } \frac{Z}{\sqrt 3} + \sum_{ k = 1}^\infty \alpha_k U_k,
\end{align*}
where $Z$ is standard Gaussian independent of  $U_1,U_2,\ldots$ which are i.i.d.\ $\Uni[-1,1]$, and $\alpha_1 \geq \alpha_2 \geq \ldots $ is a  non-increasing sequence of  non-negative reals with $||\alpha||_2<1$. We obtain a similar result for random projections of the uniform distribution on the discrete cube $\{-1,+1\}^n$.

%Assume that $(\Theta^{(n)})_{n\in\N}$ is a sequence of random vectors uniformly distributed on $\SSS^{n-1}$. Then $\|\Theta^{(n)}\|_\infty\to 0$ a.s. as $n\to\infty$ and so with probability $1$, the sequence
%\[
%\mu_{\Theta^{(n)}}(A) := \mathbb{P} \left[ \langle U, \Theta^{(n)} \rangle \in A \right],\qquad A\subset \R\,\,\text{Borel}, U\sim\Uni(\B_\infty^n)
%\]
%of random probability measures converges weakly to a centered Gaussian random variable with variance $1/3$. In this paper, we prove a large deviation principle at speed $n$ for the sequence of projections $\mu_{\Theta^{(n)}}$ of the uniform distribution on the cube $\B_\infty^n$ onto a random direction $\Theta^{(n)}$ with an explicit good rate function. We obtain a similar result for projections of the uniform distribution on the discrete cube $\{-1,+1\}^n$.
\medspace
\end{abstract}

%\tableofcontents

% % % % % % % % % % % % % % % % % % % % % % % % % % % % % % % % %
\section{Introduction and main results}
% % % % % % % % % % % % % % % % % % % % % % % % % % % % % % % % %

Understanding the large deviation behavior of random geometric objects and quantities that are classically studied in asymptotic geometric analysis and high-dimensional probability theory, has attracted considerable attention in the past four years. Central limit theorems had already been obtained in various situations (e.g., \cite{APT2019}, \cite{K2007}, \cite{PPZ14}, \cite{Schmu2001}), but the universality that governs the behavior as the dimension of the space tends to infinity restricts the information that can be extracted, for instance, from lower-dimensional projections of high-dimensional distributions, such as uniform distributions on high-dimensional convex bodies. This motivates the study of large deviation principles, a type of limit theorem which is sensitive to the distribution of the underlying random elements, thereby allowing us, for instance, to distinguish high-dimensional convex bodies from their lower-dimensional projections.
Another motivation stems from a recently established connection \cite{APT2020} between large (and moderate) deviation principles for isotropic log-concave random vectors and the famous Kannan-Lov\'asz-Simonovits conjecture.
Starting with the work of Gantert, Kim, and Ramanan \cite{GKR2017}, who studied large deviation principles for one-dimensional projections of $\ell_p^n$ balls in both the annealed and quenched setting, we have seen a variety of large deviation principles appear in the literature. A version for multidimensional random projections of $\ell_p^n$ balls was obtained by Alonso-Guti\'errez, Prochno, and Th\"ale in \cite{APT2018}. Large deviation principles for $\ell_q$-norms of random vectors distributed according to certain families of distributions on $\ell_p^n$ balls were proved by Kabluchko, Prochno, and Th\"ale in \cite{KPT2019_I,KPT2019_II}. Level--2 large deviations were studied by Kim and Ramanan in the setting of $\ell_p^n$ spheres in \cite{KR2018} and in the non-commutative setting of Schatten class unit balls by Kabluchko, Prochno, and Th\"ale in \cite{KPT2019_sanov}. Recently, Kim, Liao, and Ramanan \cite{KLR2019} have studied both level--1 and level--2 large deviation principles under asymptotic thin-shell conditions, which, on the one hand unveil a similarity to the central limit theorem for convex bodies due to Klartag \cite{K2007} (which follows from the `classical' thin-shell estimate) and, on the other hand, allow to extend earlier works for $\ell_p^n$ balls to the setting of Orlicz balls.

In this paper, we study the large deviation behavior of the random one-dimensional  projections of uniform distributions on the cubes $[-1,1]^n$, as $n\to\infty$. To be more precise, let $n\in\N$ and $U^{(n)} := (U_1,\ldots,U_n)$ be a random vector uniformly distributed on the $n$-dimensional unit cube $\B_\infty^n:=[-1,1]^n$. For a vector $\theta^{(n)}:=(\theta_1,\ldots,\theta_n)$ in the unit sphere $\mathbb{S}^{n-1}:=\{x\in\R^n\,:\,\|x\|_2=1 \}$ let $\mu_{\theta^{(n)}}$ be the (deterministic) probability measure on $\mathbb{R}$ defined by setting
\begin{align*}
\mu_{\theta^{(n)}}(A) := \mathbb{P} \left[ \langle U^{(n)}, \theta^{(n)} \rangle \in A \right],
\end{align*}
for every Borel set $A\subset \mathbb R$. In words, $\mu_{\theta^{(n)}}$ is the projection of the uniform distribution on the cube $\B_\infty^n$ onto the line determined by the direction $\theta^{(n)}\in\SSS^{n-1}$. Denoting by $\text{Law}(X)$  the law of a random variable $X$,  we can write
$$
\mu_{\theta^{(n)}} = \text{Law} \left(\theta_1 U_1 + \ldots+ \theta_n U_n\right).
$$
We are interested in the limit properties of $\mu_{\theta^{(n)}}$ as $n\to\infty$, where $(\theta^{(n)})_{n\in \mathbb N}$ is a sequence of directions  with $\theta^{(n)} \in \mathbb S^{n-1}$ for $n\in\N$.
For example, if $\theta^{(n)} = (\frac 1 {\sqrt n}, \ldots, \frac 1 {\sqrt n})$, $n\in\N$, then the classical central limit theorem implies the weak convergence
\[
\mu_{\theta^{(n)}} \xrightarrow[n\to\infty]{w}\mathcal N(0,1/3),
\]
since $\E[U_1]=0$ and $\Var[U_1]=1/3$. More generally, the Lindeberg central limit theorem implies that the same conclusion holds for every sequence of directions $(\theta^{(n)})_{n\in \mathbb N}$ provided that $\lim_{n\to\infty} \|\theta^{(n)}\|_\infty = 0$. On the other hand, the example $\theta^{(n)} = (1,0,\ldots,0)\in \mathbb S^{n-1}$ clearly shows that limit laws other than the Gaussian one are possible. Still, the weak convergence to the Gaussian law is typical in the following sense:
for every $n\in\mathbb N$ let $\Theta^{(n)}$ be a random vector uniformly distributed on the unit sphere $\mathbb{S}^{n-1}$; note that $\mu_{\Theta^{(n)}}$ is now a random probability measure. Then, it is easy to check that $\|\Theta^{(n)}\|_\infty$ converges to $0$ a.s.~and so it follows that, with probability $1$,
\[
\mu_{\Theta^{(n)}} \xrightarrow[n\to\infty]{w}\mathcal N(0,1/3).
\]
This well-known fact is a simple special case of the celebrated central limit theorem for convex bodies obtained by Klartag~\cite{K2007}.
%.???? Mention that all the above facts are well known and  this discussion is just a simple special case of deep results on projections of convex bodies (Klartag's CLT)...
While this shows that the typical behavior is universal, our main results will  make clear that the atypical behavior is not.
 %but also provide the explicit dependency on the underlying distribution by identifying the rate function that governs the large deviation principle.

% % % % % % % % % % % % % % %
\subsection{Main results}
% % % % % % % % % % % % % % %

Our main results characterize the probabilities of atypical directions in terms of a large deviation principle. In what follows, let $\mathcal M_1(\mathbb R)$ be the set of Borel probability measures on $\mathbb R$ endowed with the weak topology.

\begin{thmalpha}\label{theo:main}
For every $n\in\mathbb N$ let $\Theta^{(n)}$ be a random vector uniformly distributed on $\mathbb{S}^{n-1}$. Then the sequence of random measures $(\mu_{\Theta^{(n)}})_{n\in \mathbb N}$ satisfies a large deviation principle on $\mathcal M_1(\mathbb R)$ at speed $n$ and with a good rate function $I: \mathcal M_1(\mathbb R)\to [0,+\infty]$ given by
\begin{align*}
I(\nu(\alpha)) = - \frac{1}{2} \log \left( 1 - \|\alpha\|_2^2 \right)
\end{align*}
whenever $\nu(\alpha)$ is equal to the law of the random variable
$$
\sqrt{1 - \|\alpha\|_2^2 }  \frac{ Z}{ \sqrt{3}}  + \sum_{ k =1}^\infty \alpha_k U_k,
$$
where $\alpha_1\geq \alpha_2\geq \ldots \geq 0$ is a non-increasing sequence of non-negative reals satisfying $\|\alpha\|_2^2 := \sum_{ k = 1}^\infty \alpha_k^2 < 1$, $Z$ is a standard Gaussian random variable, $U_1,U_2,\ldots$ are uniform random variables on $[-1,1]$, and $Z,U_1,U_2,\ldots$ are independent. Whenever $\nu$ is not of this form, $I(\nu) = + \infty$.
\end{thmalpha}

Observe that the function $I(\nu)$ vanishes if and only if $\nu$ is the law of a centered Gaussian variable with variance $1/3$, which corresponds to the typical behavior described above.

The proof we shall present for Theorem \ref{theo:main} carries over to the uniform distribution on the discrete cube $\{-1,+1\}^n$, where the independent random variables $U_1,\dots,U_n \sim \Uni[-1,1]$ are replaced by independent Rademacher random variables $V_1,\dots,V_n$, where $\Pro[V_1=+1]=\frac{1}{2}=\Pro[V_1=-1]$. Since those Rademacher random variables are centered with variance $1$, we do not obtain a Gaussian of variance $1/3$, but a standard Gaussian. In what follows, for $V^{(n)}:=(V_1,\dots,V_n)$, we shall write
\begin{align*}
\mu_{\Theta^{(n)}}^{\text{discr}}(A) := \mathbb{P} \left[ \langle V^{(n)}, \Theta^{(n)} \rangle \in A \right],\qquad A\subset \R \,\,\text{Borel},\,\, \Theta^{(n)}\sim\Uni(\SSS^{n-1})
\end{align*}
for the random probability measure corresponding to the random projection of the discrete uniform distribution on $\{-1,+1\}^n$.

\begin{thmalpha}\label{thm:main_B}
For every $n\in\mathbb N$ let $\Theta^{(n)}$ be a random vector uniformly distributed on $\mathbb{S}^{n-1}$. Then the sequence of random measures $(\mu_{\Theta^{(n)}}^{\text{discr}})_{n\in \mathbb N}$ satisfies a large deviation principle on $\mathcal M_1(\mathbb R)$ at speed $n$ and with a good rate function $I: \mathcal M_1(\mathbb R)\to [0,+\infty]$ given by
\begin{align*}
I(\nu(\alpha)) = - \frac{1}{2} \log \left( 1 - \|\alpha\|_2^2 \right)
\end{align*}
whenever $\nu(\alpha)$ is equal to the law of the random variable
$$
\sqrt{1 - \|\alpha\|_2^2 }  Z  + \sum_{ k =1}^\infty \alpha_k V_k,
$$
where $\alpha_1\geq \alpha_2\geq \ldots \geq 0$ is a non-increasing sequence of non-negative reals satisfying $\|\alpha\|_2^2 := \sum_{ k = 1}^\infty \alpha_k^2 < 1$, $Z$ is a standard Gaussian random variable, $V_1,V_2,\ldots$ are Rademacher random variables, and $Z,V_1,V_2,\ldots$ are independent. Whenever $\nu$ is not of this form, $I(\nu) = + \infty$.
\end{thmalpha}

Let us briefly explain the idea of proof of Theorem~\ref{theo:main}. The law of the random variable $\theta_1 U_1 + \ldots+ \theta_n U_n$ does not change if we replace $\theta_1,\ldots,\theta_n$ by the decreasing order statistics of the absolute values $|\theta_1|,\ldots,|\theta_n|$. Therefore, our task reduces essentially to establishing a large deviation principle for the order statistics of  $|\Theta_1^{(n)}|,\ldots, |\Theta_n^{(n)}|$.  The reduction to this problem is justified in Section~\ref{subsec:passing}, while the LDP for the order statistics is established in Section~\ref{subsec:LDP_order}.

% % % % % % % % % % % % % % % % % % % % %
\section{Notation and Preliminaries}
% % % % % % % % % % % % % % % % % % % % %

Let us briefly recall (and complement) the basic notation used throughout this paper. If $n\in\N$, then $\mathbb{S}^{n-1}:=\{x\in\R^n\,:\,\|x\|_2=1 \}$ is the Euclidean unit sphere, and the cube in $\R^n$ is denoted by $\B_\infty^n:=[-1,1]^n$. The standard inner product on $\R^n$ is denoted by $\langle \cdot, \cdot\rangle$. For a Borel measurable set $A\subset \R^n$, we denote by $\vol_n(A)$ its $n$-dimensional Lebesgue measure. For a set $A\subset\R^n$, we denote by $A^\circ$ and $\overline{A}$ its interior and closure respectively.

Let us continue with some notions and results from large deviation theory. For a thorough introduction to this topic, we refer the reader to \cite{DZ2010}.

\begin{df}
Let $(\xi_n)_{n\in\N}$ be a sequence of random elements taking values in some metric space $M$. Further, let $(s_n)_{n\in\N}$ be a sequence of positive reals
%positive sequence
with $s_n\uparrow\infty$ and $\mathcal{I}:M\to[0,+\infty]$ be a lower semi-continuous function.
We say that $(\xi_n)_{n\in\N}$ satisfies a (full) large deviations principle (LDP) with speed $s_n$ and a rate function $\mathcal{I}$ if
\begin{equation}\label{eq:LDPdefinition}
\begin{split}
-\inf_{x\in A^\circ}\mathcal{I}(x)
\leq\liminf_{n\to\infty}{1\over s_n}\log\Pro\left[\xi_n \in A \right]
\leq\limsup_{n\to\infty}{1\over s_n}\log\Pro\left[\xi_n \in A \right]\leq -\inf_{x\in\overline{A}}\mathcal{I}(x)
\end{split}
\end{equation}
for all Borel sets $A\subset M$. The rate function $\mathcal I$ is called good if its lower level sets
%level sets
$\{x\in M\,:\, \mathcal{I}(x) \leq \alpha \}$ are compact for all finite $\alpha\geq 0$.
We say that $(\xi_n)_{n\in\N}$ satisfies a weak LDP with speed $s_n$ and rate function $\mathcal{I}$ if the rightmost upper bound in \eqref{eq:LDPdefinition} is valid only for compact sets $A\subset M$.
\end{df}

What separates a weak from a full LDP is the so-called exponential tightness of the sequence of random variables (see, e.g., \cite[Lemma 1.2.18]{DZ2010}).

\begin{proposition}\label{prop:equivalence weak and full LDP}
Let $(\xi_n)_{n\in\N}$ be a sequence of random elements taking values in $M$. Suppose that it satisfies a weak LDP with speed $s_n$ and rate function $\mathcal{I}$. Then $(\xi_n)_{n\in\N}$ satisfies a full LDP if and only if the sequence is exponentially tight, that is, if and only if for every $C\in(0,\infty)$ there exists a compact set $K_C\subset M$ such that
$$
\limsup_{n\to\infty}{1\over s_n}\log \Pro\left[\xi_n\notin K_C\right]<-C\,.
$$
In this case, the rate function $\mathcal I$ is good.
\end{proposition}

The following result (see, e.g., \cite[Theorem 4.1.11]{DZ2010}) shows that to prove a weak LDP it is sufficient to consider a  base of the underlying topology on a metric space.

\begin{proposition}\label{prop:basis topology}
Let $\mathcal T$ be a base of the topology in a metric space $M$. Let $(\xi_n)_{n\in\N}$ be a sequence of $M$-valued random elements and assume $s_n\uparrow\infty$.  If for every $w\in M$,
\[
\mathcal I(w)
=
- \inf_{A\in\mathcal T:\, w\in A} \limsup_{n\to\infty} \frac 1 {s_n} \log \Pro \left[\xi_n \in A\right]
=
- \inf_{A\in\mathcal T:\, w\in A} \liminf_{n\to\infty} \frac 1 {s_n} \log \Pro \left[\xi_n \in A\right],
\]
then $(\xi_n)_{n\in\N}$ satisfies a weak LDP with speed $s_n$ and rate function $\mathcal{I}$.
\end{proposition}

% % % % % % % % % % % % % % % % % % % % % %
\section{Proof of Theorem \ref{theo:main}}
% % % % % % % % % % % % % % % % % % % % % %

% % % % % % % % % % % % % % % % % % % % % % % % %
\subsection{Passing to the space of sequences}\label{subsec:passing}
% % % % % % % % % % % % % % % % % % % % % % % % %
Consider the set
\[
\mathcal W := \Big\{\alpha = (\alpha_1,\alpha_2,\ldots)\in \mathbb R^\infty: \alpha_1\geq \alpha_2 \geq \ldots \geq 0, \|\alpha\|_2 \leq 1\Big\}
\]
endowed with the topology of coordinatewise convergence.  Using Fatou's lemma, one checks that the function $\alpha \mapsto \|\alpha\|_2$ is lower-semicontinuous on $\mathcal W$. However, it is not continuous. To see this, consider the elements $\alpha^{(n)}:= (1/\sqrt n,\ldots, 1/\sqrt n, 0,0,\ldots)$ (with $n$ non-zero terms) which converge to $0$ in $\mathcal W$ and satisfy $\|\alpha^{(n)}\|_2 = 1$. Moreover, combining Tikhonov's theorem with Fatou's lemma,  it is easy to check that $\mathcal W$ is a compact (metrizable) space.

Recall that $\mathcal M_1(\mathbb R)$ denotes  the set of Borel probability measures on $\mathbb R$ endowed with the weak topology. For every $\alpha \in \mathcal W$ consider the probability measure
\begin{equation}\label{eq:nu}
\nu(\alpha) := \text{Law} \left(\sqrt{1 - \|\alpha\|_2^2 }  \frac{ Z}{ \sqrt{3}}  + \sum_{k =1}^\infty \alpha_k U_k\right)\in \mathcal M_1(\mathbb R),
\end{equation}
where $Z,U_1,U_2,\ldots$ are as in Theorem~\ref{theo:main}.
Note that the series $\sum_{k=1}^\infty \alpha_k U_k$ converges a.s.\ and in $L^2$ as a consequence of the assumption $\|\alpha\|_2\leq 1$.  Also observe that the expectation of every probability measure $\nu(\alpha)$ is $0$, while the variance is $1$.

Let $\mathcal K:= \{\nu(\alpha): \alpha\in \mathcal W\}$ be the set of all probability measures on $\mathbb R$ which can be represented in the form $\nu(\alpha)$ for some sequence $\alpha\in \mathcal W$. We endow $\mathcal K$ with the topology of weak convergence of probability measures inherited from $\mathcal M_1(\mathbb R)$. The next proposition is the first main ingredient in the proof of Theorem~\ref{theo:main}.

\begin{proposition}\label{prop:homeomorphism}
The map $\mathcal W\ni\alpha \mapsto \nu(\alpha)\in \mathcal K$ defines a homeomorphism between the topological spaces  $\mathcal W$ and $\mathcal K$. In particular, $\mathcal K$ is compact in the weak topology.
\end{proposition}

We prove this proposition in the following two lemmas. First we show that the map $\alpha \mapsto \nu(\alpha)$ is a bijection between $\mathcal W$ and $\mathcal K$.

\begin{lemma}\label{lem:equality nu_alpha nu_beta}
If $\nu(\alpha) = \nu(\beta)$ for some $\alpha=(\alpha_k)_{k\in\N} \in \mathcal W$ and $\beta= (\beta_k)_{k\in\N}\in \mathcal W$, then $\alpha = \beta$.
\end{lemma}
\begin{proof}
We use characteristic functions. Since the series $\sum_{k=1} \alpha_k U_k$ converges a.s., L\'evy's continuity theorem implies that the characteristic function of $\nu(\alpha)$ is given by
\[
\varphi(t; \alpha) := \mathbb E  \Bigg[\eee^{\ii t \left(\sqrt{1 - \|\alpha\|_2^2 }  \frac{ Z}{ \sqrt{3}}  + \sum_{k =1}^\infty \alpha_k U_k\right)}\Bigg]
=
\eee^{-(1-\|\alpha\|_2^2)t^2/6}\prod_{k=1}^\infty \frac{\sin (\alpha_k t)}{\alpha_k t}
=
\eee^{-(1-\|\alpha\|_2^2)t^2/6}\prod_{k=1}^\infty \sinc (\alpha_k t),
\]
with the standard notation $\sinc(x) := \frac{\sin x}{ x}$. The map $x\mapsto \log \sinc x$ (together with $\log \sinc 0:=0$) defines an analytic function of the complex argument $x$ in the disc $\{x\in \mathbb C: |x|<\pi\}$ because $\sinc$ does not vanish in that disc. While the Taylor series for $\log \sinc$ could be expressed through Riemann's Zeta function, we shall only need the first term, where we have
\[
\log \sinc(x) = -\frac{x^2}{6} + O(x^4).
\]
%In fact, the Taylor series $\log \frac{\sin(x)}{x} = -\frac{x^2}{6} + O(x^4)$,
as $x\to 0$. Combining this with the square summability of $\alpha$, implies that the product on the right-hand side converges uniformly on compact sets of $\mathbb C$ and defines an analytic function of a complex variable $t\in \mathbb C$.
Similar conclusions hold for the characteristic function of $\nu(\beta)$, denoted by $\varphi(t;\beta)$. Assume now that $\varphi(t;\alpha)= \varphi(t;\beta)$ for all $t\in\mathbb R$ and hence also all $t\in \mathbb C$. Since the smallest positive real zero of $\varphi(t;\alpha)$, respectively $\varphi(t;\beta)$,  is given by $t=\pi/\alpha_1$, respectively $t= \pi/\beta_1$, we conclude that $\alpha_1= \beta_1$. Dividing the analytic functions $\varphi(t;\alpha)$ and $\varphi(t;\beta)$ by $\sin(\alpha_1t)/(\alpha_1 t)$ and $\sin (\beta_1 t)/(\beta_1 t)$, respectively, and comparing the smallest positive zeros of the resulting analytic functions, we conclude that $\alpha_2=\beta_2$. Continuing to argue in the same way yields that $\alpha= \beta$.
\end{proof}

\begin{lemma}\label{lem:continuity}
Let  $\alpha, \alpha^{(1)},\alpha^{(2)},\ldots\in \mathcal W$. Then, $\alpha^{(n)} \to \alpha$ in $\mathcal W$ if and only if $\nu(\alpha^{(n)}) \to \nu(\alpha)$ weakly.
\end{lemma}
\begin{proof}
Assume that $\alpha^{(n)}\to \alpha$ in $\mathcal W$, i.e., we assume the coordinatewise convergence $\lim_{n\to\infty}\alpha^{(n)}_k = \alpha_k$ for every $k\in\N$. As before, we use a characteristic function approach. As in the proof of Lemma \ref{lem:equality nu_alpha nu_beta}, the characteristic function of the probability measure $\nu(\alpha)$ is given by
\[
\varphi(t)
:=
\eee^{-(1-\|\alpha\|_2^2)t^2/6}\prod_{k=1}^\infty \frac{\sin (\alpha_k t)}{\alpha_k t}
=
\eee^{-(1-\|\alpha\|_2^2)t^2/6}\prod_{k=1}^\infty \sinc (\alpha_k t) ,\qquad t\in\R.
\]
Accordingly, the characteristic function of $\alpha^{(n)}$ is
\[
\varphi_n(t)
:=
\eee^{-(1-\|\alpha^{(n)}\|_2^2)t^2/6}\prod_{k=1}^\infty \frac{\sin (\alpha^{(n)}_k t)}{\alpha^{(n)}_k t},\qquad t\in\R.
\]
Fix some  $t\in\R$. Our aim is to show that $\lim_{n\to\infty} \varphi_n(t) = \varphi(t)$.  For any fixed $L\in\N$, we clearly have
\[
\lim_{n\to\infty} \prod_{k=1}^L \sinc (\alpha^{(n)}_k t)  = \prod_{k=1}^L \sinc (\alpha_k t),
\]
because of the coordinatewise convergence.  We shall now analyze the remaining part of the product. Since we would like to take logarithms, the following arguments assume that $L\geq L_0(t)$ and $n\geq n_0(t)$ are sufficiently large to ensure that the remaining terms of the product are not $0$.  More precisely, we choose $L_0(t)\in\N$ such that $\alpha_k |t| < \frac{9}{10}\pi$ for $k= L_0(t)+1$ and hence for all $k\geq L_0(t)+1$. Moreover, because of the coordinate convergence we find $n_0(t)\in\N$ such that $\alpha_k^{(n)}|t|< \frac {9}{10}\pi$ for all $n\geq n_0(t)$ and $k\geq L_0(t)+1$.
Then, the logarithm
\begin{equation}\label{eq:finite_prodicts_conv}
\log \prod_{k=L+1}^\infty \sinc (\alpha^{(n)}_k t) = \sum_{k=L+1}^\infty \log \sinc(\alpha_k^{(n)}t)
\end{equation}
is well defined. The map $x\mapsto \log \sinc x$ (together with $\log \sinc 0:=0$) defines an analytic function of the complex argument $x$ in the disc $\{x\in \mathbb C: |x|<\pi\}$ because $\sinc$ does not vanish in that disc. As in the proof of Lemma \ref{lem:equality nu_alpha nu_beta}, we shall only need the first term of the Taylor expansion, where we have
\[
\log \sinc(x) = -\frac{x^2}{6} + O(x^4).
\]
We shall apply this expansion on the disc of radius $\frac{9}{10}\pi$ around the origin, where the $O$-term is uniform. Therefore, we see that for all $L\geq L_0(t)$ and $n\geq n_0(t)$,
\begin{align*}
\sum_{k=L+1}^\infty \log \sinc(\alpha_k^{(n)}t)
& =
\sum_{k=L+1}^\infty \left(-\frac{(\alpha_k^{(n)})^2t^2}{6} + O\big((\alpha_k^{(n)})^4t^4\big)\right)\cr
&=
-\frac{t^2}{6}\Big[\|\alpha^{(n)}\|_2^2 - \sum_{k=1}^L(\alpha_k^{(n)})^2\Big] + O\Big(\sum_{k=L+1}^\infty (\alpha_k^{(n)})^4\Big).
\end{align*}
Here and everywhere, the constant implicit in the $O$-term is uniform in $L\geq L_0(t)$ and $n\geq n_0(t)$. This implies that
\begin{align*}
\log \Bigg(\eee^{-(1-\|\alpha^{(n)}\|_2^2)t^2/6}\prod_{k=L+1}^\infty \sinc(\alpha^{(n)}_k t) \Bigg) & =  -\frac{t^2}{6}\Big[1-\|\alpha^{(n)}\|_2^2\Big] -\frac{t^2}{6}\Big[\|\alpha^{(n)}\|_2^2 - \sum_{k=1}^L(\alpha_k^{(n)})^2\Big] + O\Big(\sum_{k=L+1}^\infty (\alpha_k^{(n)})^4\Big) \\
& = -\frac{t^2}{6}\Big[1-\sum_{k=1}^L(\alpha_k^{(n)})^2\Big] + O\Big(\sum_{k=L+1}^\infty (\alpha_k^{(n)})^4\Big).
\end{align*}
The same computations, with $\alpha^{(n)}$ replaced by $\alpha$, yield
\[
\log \Bigg(\eee^{-(1-\|\alpha\|_2^2)t^2/6}\prod_{k=L+1}^\infty \sinc(\alpha_k t)\Bigg) = -\frac{t^2}{6}\Big[1-\sum_{k=1}^L\alpha_k^2\Big] + O\Big(\sum_{k=L+1}^\infty \alpha_k^4\Big).
\]
Using the previous observations, for any choice of $L\geq L_0(t)$ and $n\geq n_0(t)$, we obtain
\begin{multline*}
%\lefteqn{\log \frac{\varphi(t)}{\prod_{k=1}^L \sinc (\alpha^{(n)}_k t)} - \log \frac{\varphi_n(t)}{\prod_{k=1}^L \sinc (\alpha_k t)}}\\
\log \frac{\eee^{-(1-\|\alpha^{(n)}\|_2^2)t^2/6}\prod_{k=L+1}^\infty \sinc(\alpha^{(n)}_k t)}{ \eee^{-(1-\|\alpha\|_2^2)t^2/6}\prod_{k=L+1}^\infty \sinc(\alpha_k t)}\\
%& = -\frac{t^2}{6}(1-\|\alpha\|_2^2) + \log \prod_{k=1}^L \frac{\sin (\alpha_k t)}{\alpha_k t} + \log \prod_{k=L+1}^\infty \frac{\sin (\alpha_k t)}{\alpha_k t} \cr
%& \quad +\frac{t^2}{6}(1-\|\alpha^{(n)}\|_2^2) - \log \prod_{k=1}^L \frac{\sin (\alpha_k^{(n)} t)}{\alpha_k^{(n)} t} - \log \prod_{k=L+1}^\infty \frac{\sin (\alpha_k^{(n)} t)}{\alpha_k^{(n)} t} \cr
 =
\frac{t^2}{6}\Big[1-\sum_{k=1}^L\alpha_k^2\Big] -\frac{t^2}{6}\Big[1-\sum_{k=1}^L(\alpha_k^{(n)})^2\Big] + O\Big(\sum_{k=L+1}^\infty \alpha_k^4\Big) +O\Big(\sum_{k=L+1}^\infty (\alpha_k^{(n)})^4\Big).
\end{multline*}
Now, it follows from the coordinatewise convergence that
\[
\lim_{n\to\infty} \Bigg(\frac{t^2}{6}\Big[1-\sum_{k=1}^L\alpha_k^2\Big] -\frac{t^2}{6}\Big[1-\sum_{k=1}^L(\alpha_k^{(n)})^2\Big]\Bigg) = 0.
\]
Since $\|\cdot\|_2$ dominates $\|\cdot\|_4$ and because $\alpha\in\ell_2$ is square-summable, we have
\[
\sum_{k=L+1}^\infty \alpha_k^4 \leq \Big(\sum_{k=L+1}^\infty \alpha_k^2\Big)^2 \stackrel{L\to\infty}{\longrightarrow} 0.
\]
On the other hand, since the coordinates of  $\alpha^{(n)}$ are non-increasing and $\|\alpha^{(n)}\|_2\leq 1$, if $L\geq L_0(t)$ is sufficiently large, then
\[
\sum_{k=L+1}^\infty (\alpha_k^{(n)})^4 \leq \sum_{k=L+1}^\infty (\alpha_{L+1}^{(n)})^2(\alpha_k^{(n)})^2 = (\alpha_{L+1}^{(n)})^2 \sum_{k=L+1}^\infty (\alpha_k^{(n)})^2 \leq (\alpha_{L+1}^{(n)})^2.
\]
Therefore, we obtain
\[
\lim_{L\to\infty} \limsup_{n\to\infty} \sum_{k=L+1}^\infty (\alpha_k^{(n)})^4
\leq
\lim_{L\to\infty} \limsup_{n\to\infty}(\alpha_{L+1}^{(n)})^2
=
\lim_{L\to\infty} \limsup_{n\to\infty}(\alpha_{L+1})^2
=
0.
\]
Taking everything together, we arrive at the following claim: for every $\eps\in(0,\infty)$ there exists a sufficiently large $L(\eps)\in \N$ such that
$$
-\eps
\leq
\liminf_{n\to\infty} \log  \frac{\eee^{-(1-\|\alpha^{(n)}\|_2^2)t^2/6}\prod_{k=L+1}^\infty \sinc(\alpha^{(n)}_k t)}
{\eee^{-(1-\|\alpha\|_2^2)t^2/6}\prod_{k=L+1}^\infty \sinc(\alpha_k t)}
\leq
\limsup_{n\to\infty} \log  \frac{\eee^{-(1-\|\alpha^{(n)}\|_2^2)t^2/6}\prod_{k=L+1}^\infty \sinc(\alpha^{(n)}_k t)}
{\eee^{-(1-\|\alpha\|_2^2)t^2/6}\prod_{k=L+1}^\infty \sinc(\alpha_k t)}
\leq
\eps.
$$
%Exponentiating and taking into account~\eqref{eq:finite_prodicts_conv},
Exponentiating and taking~\eqref{eq:finite_prodicts_conv} into account, we arrive at
$$
\eee^{-\eps}
\leq
\liminf_{n\to\infty}  \frac{\eee^{-(1-\|\alpha^{(n)}\|_2^2)t^2/6}\prod_{k=1}^\infty \sinc(\alpha^{(n)}_k t)}
{\eee^{-(1-\|\alpha\|_2^2)t^2/6}\prod_{k=1}^\infty \sinc(\alpha_k t)}
\leq
\limsup_{n\to\infty}  \frac{\eee^{-(1-\|\alpha^{(n)}\|_2^2)t^2/6}\prod_{k=1}^\infty \sinc(\alpha^{(n)}_k t)}
{\eee^{-(1-\|\alpha\|_2^2)t^2/6}\prod_{k=1}^\infty \sinc(\alpha_k t)}
\leq
\eee^{\eps}.
$$
Since $\eps\in(0,\infty)$ is arbitrary, both limits are, in fact, equal to $1$. This shows that $\lim_{n\to\infty}  \varphi_n(t) = \varphi(t)$ for all $t\in\R$ and completes the proof that $\nu(\alpha^{(n)}) \to \nu(\alpha)$ weakly as $n\to\infty$.

To prove the converse direction of the lemma, assume that $\alpha^{(n)}\in \mathcal W$ and $\alpha\in \mathcal W$ are such that $\nu(\alpha^{(n)}) \to \nu(\alpha)$ weakly as $n\to\infty$. Assume, by contraposition, that $\alpha^{(n)}$ does not converge to $\alpha$ in $\mathcal W$. Then, by compactness of $\mathcal W$ we can extract a subsequence of $\alpha^{(n_k)}$ converging to $\beta\neq \alpha$ as $k\to\infty$. Applying the first part of the lemma, we deduce that $\nu(\alpha^{(n_k)})$ converges to $\nu(\beta)$ weakly as $k\to\infty$. On the other hand, the same sequence converges weakly to $\nu(\alpha)$. Hence, $\nu(\alpha)=\nu(\beta)$. By Lemma~\ref{lem:equality nu_alpha nu_beta} we must have $\alpha=\beta$, which is a contradiction.
\end{proof}

We can now present the proof of Proposition~\ref{prop:homeomorphism}, establishing the homeomorphism between the topological spaces  $\mathcal W$ and $\mathcal K$.

\begin{proof}[Proof of Proposition~\ref{prop:homeomorphism}]
Lemma~\ref{lem:equality nu_alpha nu_beta} implies that the map $\alpha\mapsto \nu(\alpha)$ is a bijection between $\mathcal W$ and $\mathcal K$. Lemma~\ref{lem:continuity} implies that both, this map and its inverse, are continuous. In particular, the compactness of $\mathcal W$ and the continuity imply the compactness of $\mathcal K$ in the weak topology.
\end{proof}

% % % % % % % % % % % % % % % % % % % % % % % % % % % % % %
\subsection{Large deviations for the order statistics} \label{subsec:LDP_order}
% % % % % % % % % % % % % % % % % % % % % % % % % % % % % %

%Let $g_1,g_2,\ldots$ be independent standard normal random variables. For $n\in\N$, let $S_n^2:= g_1^2 + \ldots +g_n^2$. Then, the vector $(g_1/S_n, \ldots, g_n/S_n)$ is uniformly distributed on the unit sphere $\mathbb S^{n-1}$. We denote by $g_{1;n}\geq  \ldots \geq  g_{n;n}\geq 0$ the decreasing order statistics of the sample $|g_1|,\ldots, |g_n|$ and put $g_{k;n}:=0$ for $k>n$. Finally, define
Let  $\Theta^{(n)} =(\Theta_1^{(n)},\ldots,\Theta_n^{(n)})$ be uniformly distributed on the unit sphere $\mathbb S^{n-1}$. We denote by $\Theta_{1:n}\geq  \ldots \geq  \Theta_{n:n}\geq 0$ the decreasing order statistics of the sample $|\Theta^{(n)}_1|,\ldots, |\Theta^{(n)}_n|$ and put $\Theta_{k:n}:=0$ for $k>n$. Finally, define
$$
\eta_n:= \left(\Theta_{1:n}, \ldots, \Theta_{n:n}, 0,0,\ldots \right),
$$
which is a random element of the space $\mathcal W$.
Thus, $\eta_n$ is the vector of the decreasing order statistics of the absolute values of the coordinates of a random vector that is uniformly distributed on $\mathbb S^{n-1}$. Recall the definition of the map $\nu: \mathcal W \to \mathcal K \subset \mathcal M_1(\R)$ from~\eqref{eq:nu}. Observe that $\nu(\eta_n) = \mu_{\Theta^{(n)}}$ is the random probability measure we are interested in Theorem~\ref{theo:main}. Since $\nu$ is a homeomorphism between $\mathcal W$ and $\mathcal K$ by Proposition~\ref{prop:homeomorphism}, the proof of Theorem~\ref{theo:main} is a consequence of the following large deviation principle for $(\eta_n)_{n\in\N}$.

%The second main ingredient in the proof of Theorem~\ref{theo:main} is

\begin{proposition}\label{prop:ldp on W}
The sequence $(\eta_n)_{n\in \mathbb N}$ satisfies a large deviation principle on the compact space $\mathcal W$ equipped the topology of coordinatewise convergence at speed $n$ and with good rate function
\[
\mathcal J(\alpha) = -\frac 12 \log \big(1-\|\alpha\|_2^2\big)\in [0,+\infty], \qquad \alpha \in \mathcal W.
\]
\end{proposition}
Closely related results can be found in~\cite[Theorems~3.4 and~3.7]{barthe_etal}.
Before proving Proposition \ref{prop:ldp on W} we start with a lemma.
%We shall start with a lemma before we enter the proof of the proposition.

%\begin{lemma}\label{lem:asymptotic of multivariate beta normalization}
%Let $\ell,n\in\N$ with $\ell\leq n$ and consider $y:=(\frac{1}{2},\dots,\frac{1}{2},\frac{n-\ell}{2})\in\R^{\ell+1}$. Then, as $n\to\infty$,
%\[
%\Beta(y) \sim (2 \pi)^{\frac{\ell}{2}} e^{-\frac{\ell}{2}\log(n-\ell-2)},
%\]
%where $\Beta(\cdot)$ denotes the multivariate beta function.
%\end{lemma}
%\begin{proof}
%From the definition of the multivariate beta function, we obtain
%\[
%\Beta(y) = \frac{\prod_{i=1}^{\ell+1}\Gamma(y_i)}{\Gamma(\sum_{i=1}^{\ell+1}y_i)} =  %\frac{\pi^{\frac{\ell}{2}}\Gamma(\frac{n-\ell}{2})}{\Gamma(\frac{n}{2})}.
%\]
%As a consequence of Stirling's formula, as $n\to\infty$,
%\[
%\frac{\Gamma(\frac{n-\ell}{2})}{\Gamma(\frac{n}{2})} \sim %\sqrt{1-\frac{\ell}{n-2}}e^{-\frac{\ell}{2}\log(n-\ell-2)}(2e)^{\frac{\ell}{2}}\Big(1-\frac{\ell}{n-2}\Big)^{\frac{n-2}{2}}.
%\]
%This leads to the desired asymptotic behavior.
%\end{proof}

\begin{lemma}\label{lem:density}
For $\ell \leq n-1$, the density of the first $\ell$ coordinates $(\Theta_1^{(n)},\ldots,\Theta_\ell^{(n)})$ of a random vector $\Theta^{(n)} =(\Theta_1^{(n)},\ldots,\Theta_n^{(n)})$ uniformly distributed on the unit sphere $\mathbb S^{n-1}$ is given by
\begin{align*}
f(s_1,\ldots,s_\ell)
:=
\frac{\Gamma(\frac{n}{2}) }{\pi^{\ell/2}  \Gamma\left( \frac{n- \ell}{2} \right)} \left( 1 - s_1^2 - \ldots - s_\ell^2\right)^{ \frac{n-\ell-2}{2}  } ,\qquad
s_1^2+\ldots+s_\ell^2 \leq 1.
\end{align*}
\end{lemma}
\begin{proof}
This is well-known and follows, for example, from~\cite[Lemma~3.1]{beta_polytopes} with $\beta=-1$.
\end{proof}

\begin{proof}[Proof of Proposition \ref{prop:ldp on W}]
To prove this result, we shall use Proposition \ref{prop:basis topology} and work on a base of the topology of coordinatewise convergence on the compact space $\mathcal W$. The compactness of $\mathcal W$ is particularly sufficient to guarantee a full LDP (rather than the weak LDP coming from the proposition).
For $r\in(0,\infty)$, $\ell\in\N$, and $x \in \R^\ell$  we define an $\ell$-dimensional ball
\[
B_{r,\ell}(x) := \Big\{ z = (z_1,\ldots,z_\ell)\in \R^\ell\,:\, \sum_{i=1}^\ell (z_i - x_i)^2 < r^2 \Big\} \subset \R^\ell.
\]
A base of the topology of $\mathcal W$ is given by the following family of sets:
\[
\mathcal B:= \Big\{W_{r,\ell}(x)\,:\, r\in(0,\infty),\,\ell\in\N,\,x= (x_1,\ldots,x_\ell) \in \R^\ell,x_1\geq \ldots\geq x_\ell \geq 0,x_1^2+\ldots+x_\ell^2 \leq 1 \Big\},
\]
where
$$
W_{r,\ell}(x) = \big\{w\in \mathcal W: (w_1,\ldots,w_\ell)\in B_{r,\ell}(x)\big\}.
$$
\vskip 1mm
\textbf{Lower bound:}
Let us prove that for every $x\in \R^\ell$ with $\|x\|_2<1$ and $x_1\geq \ldots\geq x_\ell\geq 0$ we have
\begin{equation}\label{eq:lower_bound_x}
\liminf_{n\to\infty} \frac{1}{n}\log
\Pro\Big[\eta_n \in W_{r,\ell}(x)\Big] \geq \frac{1}{2}\log\big(1-\|x\|_2^2\big).
\end{equation}
Without loss of generality, we may assume that $x_\ell>0$ because any ball $B_{r,\ell}(x)$ contains a smaller ball $B_{r',\ell}(x')$ with $x'\in\R^\ell$ satisfying this condition. So, let $x_\ell>0$.  Observe that
\begin{align*}
\Pro\Big[\eta_n \in W_{r,\ell}(x)\Big]
& =
\Pro\Bigg[ (\Theta_{1:n}, \ldots, \Theta_{\ell:n})\in B_{r,\ell}(x)\Bigg] \cr
& \geq
\Pro\Bigg[
(\Theta^{(n)}_1, \ldots, \Theta^{(n)}_\ell) \in B_{r,\ell}(x),
\quad
\Theta^{(n)}_1 > \ldots >\Theta^{(n)}_\ell>0,
\quad
\max\left\{|\Theta^{(n)}_{\ell+1}|, \ldots, |\Theta^{(n)}_{n}|\right\} < \Theta^{(n)}_\ell
 \Bigg]\\
& = \int_{B_{r,\ell}^+(x)} f(z_1,\ldots,z_\ell) \,\Pro\Bigg[ \max\left\{|\Theta^{(n)}_{\ell+1}|, \ldots, |\Theta^{(n)}_{n}|\right\} < z_\ell \Big | \Theta^{(n)}_1 = z_1,\ldots, \Theta^{(n)}_\ell = z_\ell \Bigg]  \dint z_1\ldots \dint z_\ell,
\end{align*}
where $f$ is the joint density of $(\Theta^{(n)}_1, \ldots, \Theta^{(n)}_\ell)$ given in Lemma~\ref{lem:density} and
$$
B_{r,\ell}^+(x) := B_{r,\ell}(x) \cap \big\{z\in \mathbb R^\ell: z_1>\ldots > z_\ell >0, z_1^2+\ldots +z_\ell^2 \leq 1\big\}.
$$
Now, the conditional distribution of $\Theta^{(n)}_{\ell+1}, \ldots, \Theta^{(n)}_{n}$ given that $\Theta^{(n)}_1 = z_1,\ldots, \Theta^{(n)}_\ell = z_\ell$ is the uniform distribution on the $(n-\ell-1)$-dimensional sphere of radius $(1-z_1^2-\ldots -z_\ell^2)^{1/2}\leq 1$. Introducing a random vector $Y^{(n-\ell)} = (Y_1^{(n-\ell)},\ldots, Y_{n-\ell}^{(n-\ell)})$ distributed uniformly on the unit sphere $\mathbb S^{n-\ell-1}$, we can continue as follows:

\begin{align*}
\Pro\Big[\eta_n \in W_{r,\ell}(x)\Big]
& \geq
\int_{B_{r,\ell}^+(x)} f(z_1,\ldots,z_\ell) \,\Pro\Bigg[ \sqrt{1-z_1^2-\ldots- z_\ell^2} \max\left\{|Y_1^{(n-\ell)}|, \ldots, |Y_{n-\ell}^{(n-\ell)}|\right\} < z_\ell \Bigg]  \dint z_1\ldots \dint z_\ell\\
& \geq
\int_{B_{1/n,\ell}^+(x)} f(z_1,\ldots,z_\ell) \,\Pro\Big[\max\left\{|Y_1^{(n-\ell)}|, \ldots, |Y_{n-\ell}^{(n-\ell)}|\right\} < x_\ell/2 \Big]  \dint z_1\ldots \dint z_\ell,
\end{align*}
where in the last step we used that, for sufficiently large $n$, we have $1/n \leq r$ and $z_\ell > x_\ell/2>0$ for all $(z_1,\ldots,z_\ell)\in B_{1/n,\ell}^+(x)$. Now, it is well known that for every fixed $x_\ell >0$, we have
$$
\lim_{n\to\infty} \Pro\Big[\max\left\{|Y_1^{(n-\ell)}|, \ldots, |Y_{n-\ell}^{(n-\ell)}|\right\} < x_\ell/2 \Big] = 1.
$$
(To prove this claim, write $Y^{(n)}= G^{(n)}/\|G^{(n)}\|_2$, where $G^{(n)}$ is standard Gaussian on $\mathbb R^{n}$, and recall that $\|G^{(n)}\|_\infty \sim \sqrt{2\log n}$ a.s., as $n\to\infty$,  while $\|G^{(n)}\|_2 \sim \sqrt n$ a.s., implying that $\|Y^{(n)}\|_\infty =\|G^{(n)}\|_\infty/\|G^{(n)}\|_2\to 0$ in probability, as $n\to\infty$).
Hence, for sufficiently large $n$, the above  probability is $\geq 1/2$ and using the form of the density $f$ given in Lemma~\ref{lem:density}, we obtain
$$
\Pro\Big[\eta_n \in W_{r,\ell}(x)\Big] \geq \frac 12 \int_{B_{1/n,\ell}^+(x)} f(z_1,\ldots,z_\ell)\dint z_1\ldots \dint z_\ell
=
\frac{\Gamma(\frac{n}{2})}{2\pi^{\ell/2}  \Gamma\left( \frac{n- \ell}{2} \right)} \int_{B_{1/n,\ell}^+(x)}  \Big(1-\sum_{i=1}^\ell z_i^2\Big)^{\frac{n-\ell}{2}-1} \, \dint z_1\dots\dint z_\ell.
$$
If we denote the factor in front of the integral by $\kappa(n,\ell)$, then  we have $\lim_{n\to\infty} \frac 1n \log \kappa(n,\ell) =0$ because $\Gamma(\frac{n}{2})/\Gamma( \frac{n- \ell}{2}) \sim (\frac n2)^{\ell/2}$.
Next, we will bound the integral from below. Observe that for any $z\in B_{1/n,\ell}^+(x)$ it follows from the triangle inequality that
\[
\|z\|_2 \leq \|x\|_2 + \frac{1}{n}.
\]
Thus,
\begin{align*}
\int_{B_{1/n,\ell}^+(x)} \Big(1-\sum_{i=1}^\ell z_i^2\Big)^{\frac{n-\ell}{2}-1} \, \dint z_1\dots\dint z_\ell
& \geq \int_{B_{1/n,\ell}^+(x)} \Big(1-\Big[\|x\|_2 + \frac{1}{n}\Big]^2\Big)^{\frac{n-\ell}{2}-1} \, \dint z_1\dots\dint z_\ell \cr
& = \int_{B_{1/n,\ell}^+(x)} \Big(1-\|x\|_2^2 -2n^{-1}\|x\|_2-n^{-2}\Big)^{\frac{n-\ell}{2}-1} \, \dint z_1\dots\dint z_\ell\cr
& = \vol_\ell\big(B_{1/n,\ell}^+(x)\big) \big(1-\|x\|_2^2\big)^{\frac{n-\ell}{2}-1}\Bigg(1 - \frac{2n^{-1}\|x\|_2+n^{-2}}{1-\|x\|_2^2}\Bigg)^{\frac{n-\ell}{2}-1}.
\end{align*}
To estimate the volume of the set $B_{1/n,\ell}^+(x)$, we observe that for every point $u= (u_1,\ldots,u_{\ell}) \in \B_{1/n,\ell}(0)$ with  $u_1>\ldots > u_\ell >0$ the shifted point $z:=(u_1+ x_1,\ldots, u_\ell  + x_\ell)$ belongs to $B_{1/n,\ell}^+(x)$, provided $n$ is sufficiently large to ensure that $\|z\|_2 <1$.  Since the Lebesgue measure of the set of such $u$'s is $\vol_\ell(B_{1/n,\ell}(0))/ (2^\ell\ell!)$, we arrive at the estimate
$$
\vol_\ell\big(B_{1/n,\ell}^+(x)\big)
\geq
\frac {\vol_\ell\big(B_{1/n,\ell}(0)\big)} {2^\ell\ell!}
=
\frac {\vol_\ell\big(B_{1,\ell}(0)\big)} {2^\ell\ell! n^{\ell}}.
$$
Combining all elements, and taking the logarithm, we obtain that
\begin{align*}
\frac{1}{n}\log
\Pro\Big[\eta_n \in W_{r,\ell}(x)\Big] & \geq \frac{1}{n} \log \kappa(n,\ell) + \frac{1}{n}\log\frac {\vol_\ell\big(B_{1,\ell}(0)\big)} {2^\ell\ell! n^{\ell}}   + \frac{n-\ell-2}{2n}\log\Bigg(1 - \frac{2n^{-1}\|x\|_2+n^{-2}}{1-\|x\|_2^2}\Bigg) \cr
&\quad + \frac{n-\ell-2}{n}\frac{1}{2}\log\big(1-\|x\|_2^2\big).
\end{align*}
Taking the limit inferior, we conclude that
\begin{equation*}%\label{eq:lower_bound_x}
\liminf_{n\to\infty} \frac{1}{n}\log
\Pro\Big[\eta_n \in W_{r,\ell}(x)\Big] \geq \frac{1}{2}\log\big(1-\|x\|_2^2\big),
\end{equation*}
which completes the proof of~\eqref{eq:lower_bound_x}.

Let us now derive a lower bound from Proposition~\ref{prop:basis topology}. Take some  $w\in \mathcal W$. Our aim is to prove that
\begin{equation}\label{eq:lower_bound_as_in_prop}
\inf_{A\in\mathcal B:\, w\in A} \liminf_{n\to\infty} \frac 1 {n} \log \Pro \left[\eta_n \in A\right] \geq \frac{1}{2}\log\big(1-\|w\|_2^2\big)=: -\mathcal J(w).
\end{equation}
If $\|w\|_2=1$, the right-hand side is $-\infty$ and there is nothing to prove.  So, let in the following $\|w\|_2<1$. Let $A= W_{r,\ell}(x) \in \mathcal B$ be a set from the base of topology such that $w\in A$.  Let $\widetilde w:= (w_1,\ldots,w_\ell)\in \R^\ell$. Then $\widetilde w \in B_{r,\ell}(x)$.  This means that  $\|\widetilde w - x\|<r$ and we can find $\widetilde r>0$ such that $B_{\widetilde r, \ell}(\widetilde w) \subset B_{r,\ell} (x)$. It follows that
$$
\mathbb P[\eta_n\in W_{r,\ell}(x)] \geq  \mathbb P[\eta_n \in W_{\widetilde r,\ell} (\widetilde w)].
$$
It follows from~\eqref{eq:lower_bound_x} with $W_{r,\ell}(x)$ replaced by $W_{\widetilde r,\ell} (\widetilde w)$ that
\[
\liminf_{n\to\infty} \frac{1}{n}\log \Pro\Big[\eta_n \in W_{r,\ell}(x)\Big]
\geq
\liminf_{n\to\infty} \frac{1}{n}\log \Pro\Big[\eta_n \in W_{\widetilde r,\ell} (\widetilde w)\Big]
\geq
\frac{1}{2}\log\big(1-\|\widetilde w\|_2^2\big) \geq \frac{1}{2}\log\big(1-\|w\|_2^2\big),
\]
which proves~\eqref{eq:lower_bound_as_in_prop}.

\vskip 1mm
\textbf{Upper bound:}  Let us prove that for every $x\in \R^\ell$ with $0<r \leq \|x\|_2\leq 1$ and $x_1\geq \ldots\geq x_\ell \geq 0$ we have
\begin{equation}\label{eq:upper_bound_x}
\limsup_{n\to\infty} \frac{1}{n}\log
\Pro\Big[\eta_n \in W_{r,\ell}(x)\Big] \leq \frac{1}{2}\log\big(1- (\|x\|_2-r)^2\big).
\end{equation}
We obtain from a union bound and the exchangeability of the coordinates of $\Theta^{(n)}$ that
\begin{align*}
\Pro\Big[\eta_n \in W_{r,\ell}(x)\Big]
 =
\Pro\Bigg[ (\Theta_{1:n}, \ldots, \Theta_{\ell:n})\in B_{r,\ell}(x)\Bigg]
\leq
2^\ell
{n\choose \ell}
\Pro\Bigg[ \Big(\Theta^{(n)}_1, \ldots, \Theta^{(n)}_\ell\Big) \in B_{r,\ell}(x)\Bigg],
%& \leq \Big(\frac{en}{\ell}\Big)^\ell \Pro\Bigg[ \Big(\frac{g_{1}}{S_n}, \ldots, \frac{g_{\ell}}{S_n}\Big) \in B_{r,\ell}(x)\Bigg].
\end{align*}
where the factor of $2^\ell$ comes from the fact that each of the coordinates $\Theta_1^{(n)},\ldots,\Theta_\ell^{(n)}$ may be either positive or negative while having a certain modulus.
As before, by Lemma~\ref{lem:density} we have
\begin{align*}
\Pro\Big[\eta_n \in W_{r,\ell}(x)\Big]
&\leq
2^\ell
{n\choose \ell} \frac{\Gamma(\frac{n}{2}) }{\pi^{\ell/2}  \Gamma\left( \frac{n- \ell}{2} \right)} \int_{B_{r,\ell}(x) \cap B_{1,\ell}(0)} \Big(1-\sum_{i=1}^\ell z_i^2\Big)^{\frac{n-\ell}{2}-1} \, \dint z_1\dots\dint z_\ell\\
& = \gamma(n,\ell) \int_{B_{r,\ell}(x) \cap B_{1,\ell}(0)} \Big(1-\sum_{i=1}^\ell z_i^2\Big)^{\frac{n-\ell}{2}-1} \, \dint z_1\dots\dint z_\ell,
\end{align*}
with $\lim_{n\to\infty} \frac 1n \log \gamma(n,\ell) = 0$, so that $\gamma(n,\ell)$ will vanish on the logarithmic scale.
Next, we observe that, for any $z\in B_{r,\ell}(x)$, the reverse triangle inequality yields
\[
\|z\|_2 \geq \|x\|_2 - r \geq 0.
\]
From this we obtain the upper bound
\begin{align*}
\int_{B_{r,\ell}(x) \cap B_{1,\ell}(0)} \Big(1-\sum_{i=1}^\ell z_i^2\Big)^{\frac{n-\ell}{2}-1} \, \dint z_1\dots\dint z_\ell
& \leq \int_{B_{r,\ell}(x) \cap B_{1,\ell}(0)} \Big(1-\Big(\|x\|_2 - r\Big)^2\Big)^{\frac{n-\ell}{2}-1} \, \dint z_1\dots\dint z_\ell \cr
& = \vol_\ell\big(B_{r,\ell}(x)\big)\cdot\Big(1-\Big(\|x\|_2 - r\Big)^2\Big)^{\frac{n-\ell-2}{2}}.
%& = \vol_\ell\big(B_{r,\ell}(x)\big)\cdot \Big(1-\|x\|_2^2 +2r\|x\|_2- r^2\Big)^{\frac{n-\ell}{2}-1} \cr
%& \leq \vol_\ell\big(B_{r,\ell}(x)\big)\cdot \Big(1-\|x\|_2^2 +2r\|x\|_2\Big)^{\frac{n-\ell}{2}-1} \cr
%& = \vol_\ell\big(B_{r,\ell}(x)\big)\cdot\Big(1-\|x\|_2^2\Big)^{\frac{n-\ell-2}{2}} \Big(1+\frac{2r\|x\|_2}{1-\|x\|_2}\Big)^{\frac{n-\ell}{2}-1}
\end{align*}
Taking everything together, we arrive at
\begin{align*}
\frac{1}{n}\log \Pro\Big[\eta_n \in W_{r,\ell}(x)\Big]
\leq
 \frac{1}{n}\log \gamma(n,\ell) + \frac{1}{n}\log \vol_\ell(B_{r,\ell}(x)) + \frac{n-\ell-2}{n}\frac{1}{2}\log \Big(1-\Big(\|x\|_2 - r\Big)^2\Big).
\end{align*}
Thus, taking the limit superior as $n\to\infty$, we obtain
\[
\limsup_{n\to\infty} \frac{1}{n}\log \Pro\Big[\eta_n \in W_{r,\ell}(x)\Big] \leq \frac{1}{2}\log \big(1-(\|x\|_2 - r)^2\big).
\]

Let us now derive the upper bound from Proposition~\ref{prop:basis topology}. Take some  $w\in \mathcal W$. Our aim is to prove that
\begin{equation}\label{eq:upper_bound_as_in_prop}
\inf_{A\in\mathcal B:\, w\in A} \limsup_{n\to\infty} \frac 1 {n} \log \Pro \left[\eta_n \in A\right] \leq \frac{1}{2}\log\big(1-\|w\|_2^2\big)=: -\mathcal J(w).
\end{equation}
If $w=0$, the right-hand side is $0$ and there is nothing to prove. So, let $w\neq 0$. Take a sufficiently large $\ell\in \N$ to ensure that $\widetilde w:=(w_1,\ldots,w_\ell)$ satisfies $\|\widetilde w\|_2>0$. Then, for $A := W_{r,\ell}(\widetilde w)$ with any positive $r<\|\widetilde w\|_2$ we have
$$
\limsup_{n\to\infty} \frac{1}{n}\log
\Pro\Big[\eta_n \in A \Big] \leq \frac{1}{2}\log\big(1- (\|\widetilde w\|_2-r)^2\big).
$$
Taking first $r\to 0$ and then $\ell\to\infty$, we arrive at the required upper bound~\eqref{eq:upper_bound_as_in_prop}.

\vskip 1mm
\textbf{Completing the proof of Proposition~\ref{prop:ldp on W}.}
Putting lower and upper bounds~\eqref{eq:lower_bound_as_in_prop} and~\eqref{eq:upper_bound_as_in_prop} together, we see that for every $w\in \mathcal W$,
\[
\inf_{A\in\mathcal B:\, w\in A} \limsup_{n\to\infty} \frac{1}{n}\log \Pro\Big[\eta_n \in A\Big]
\leq
-\mathcal J(w)
\leq
\inf_{A\in\mathcal B:\, w\in A}  \liminf_{n\to\infty} \frac{1}{n}\log
\Pro\Big[\eta_n \in A\Big]
\]
and so Proposition \ref{prop:basis topology} yields a weak LDP for the sequence $(\eta_n)_{n\in\N}$ on $\mathcal W$, endowed with the topology of coordinatewise convergence, which occurs at speed $n$ and with rate function $\mathcal J$. Because the space $\mathcal W$ is compact, the exponential tightness of the probability laws of $(\eta_n)_{n\in\N}$ holds trivially. Now, the weak LDP together with exponential tightness implies the full LDP (see, e.g., \cite[Lemma 1.2.18 (a)]{DZ2010}).
\end{proof}

\subsection*{Acknowledgement}
SJ is supported by the EPSRC funded Project EP/S036202/1 \emph{Random fragmentation-coalescence processes out of equilibrium}.
ZK has been supported by the German Research Foundation under Germany’s Excellence Strategy EXC 2044 – 390685587, Mathematics M\"unster: Dynamics - Geometry - Structure and by the DFG priority program SPP 2265 Random Geometric Systems.
JP is supported by the Austrian Science Fund (FWF) Project P32405 \textit{Asymptotic Geometric Analysis and Applications} and by the FWF Project F5513-N26 which  is  a  part  of  the  Special Research  Program  \emph{Quasi-Monte  Carlo  Methods:  Theory  and  Applications}.

\bibliographystyle{plain}
\bibliography{projections}

\bigskip
\bigskip
	
	\medskip
	
	\small
	
	\noindent \textsc{Samuel Johnston:} Department of Mathematical Sciences, University of Bath, Claverton Down, Bath BA2 7AY, United Kingdom
			
		\noindent
			{\it E-mail:} \texttt{sgj22@bath.ac.uk}
		
			\medskip

	\noindent \textsc{Zakhar Kabluchko:} Faculty of Mathematics, University of M\"unster, Orl\'eans-Ring 10,
		48149 M\"unster, Germany
		
	\noindent
		{\it E-mail:} \texttt{zakhar.kabluchko@uni-muenster.de}
	
		\medskip
	
	\noindent \textsc{Joscha Prochno:} Faculty of Computer Science and Mathematics,
	University of Passau, Innstrasse 33, 94032 Passau, Germany
	
	\noindent
	{\it E-mail:} \texttt{joscha.prochno@uni-passau.de}

\end{document}